\documentclass[a4paper,12pt]{amsart}

\pdfoutput=1
\usepackage{setspace}
\usepackage[left=28mm,head=30mm,bottom=30mm,foot=20mm,right=28mm]{geometry}
\usepackage{amsmath,amsthm,amsfonts,amssymb,mathtools,mathrsfs,url,bm,enumitem,stackengine}
\usepackage{newtxtext}
\usepackage[varvw]{newtxmath}
\usepackage{graphicx}
\usepackage{tikz}
\usepackage{tikz,amsmath,amssymb}
 
\newcommand{\axesframe}{%
  \draw (0,0) rectangle (2,2);
  \draw[->] (0,0)--(2.2,0) node[right]{$\alpha$};
  \draw[->] (0,0)--(0,2.2) node[above]{$\beta$};
  \foreach \t/\l in {1/1, {5/3}/\frac53, 2/2} {
    \draw (\t,0)--(\t,-0.04) node[below,font=\scriptsize]{$\l$};
    \draw (0,\t)--(-0.04,\t) node[left,font=\scriptsize]{$\l$};
  }
}

\newcommand{\guides}{%
  \draw[dashed,gray] (1,0)--(1,2);
  \draw[dashed,gray] ({5/3},0)--({5/3},2);
  \draw[dashed,gray] (0,{5/3})--(2,{5/3});
  \draw[dashed,gray] (1,2)--(2,1.5);
}

\allowdisplaybreaks
\usepackage{hyperref}
\hypersetup{
	colorlinks=true,
	linkcolor=blue,
	citecolor=red,
}
\usepackage{appendix}

\mathtoolsset{showonlyrefs=true}
\theoremstyle{definition}

\newtheorem{remark}{Remark}[section]

\theoremstyle{plain}

\newtheorem{theorem}{Theorem}[section]

\newtheorem{lemma}{Lemma}[section]

\renewcommand{\phi}{\varphi}
\let\phi\varphi
\let\epsilon\varepsilon

\def\p{\partial}

\def\R{\mathbb R}

\def\curl{\operatorname{curl}}
\def\vu{ u }

\begin{document}

	\title[Liouville Theorems for Fractional Hall-MHD]{Liouville Theorems for the Stationary Fractional Hall-Magnetohydrodynamic Equations}

\author{Albert Ai}
\address{Beijing International Center for Mathematical Research, Peking University, No. 5 Yiheyuan Road Haidian District, Beijing 100871, People’s Republic of China}
\email{albertai@bicmr.pku.edu.cn}		
	
\author{Gast\'{o}n Vergara-Hermosilla}
\address{School of Mathematics, Harbin Institute of Technology, Xidazhi Street, Harbin 150001,
Heilongjiang, People’s Republic of China}
\email{gaston.v-h@outlook.com (corresponding author)}	

	\maketitle

\begin{abstract}
In this paper we prove Liouville-type theorems for the stationary fractional Hall-magnetohydrodynamic equations in $\mathbb R^3$. Smooth solutions vanishing at infinity, with $u\in\dot H^{\alpha/2}$ and $b\in\dot H^{\beta/2}$, are trivial for $1\le\alpha\le\frac53$, $1\le\beta<2$ and $\alpha+2\beta\le5$, with no additional integrability assumption. The key step is to estimate the Hall boundary term through the Caffarelli--Silvestre extension of $b$.
\end{abstract}

\section{Introduction and main results}

The Hall-magnetohydrodynamic (Hall-MHD) system describes the dynamics of an incompressible electrically conducting fluid when
the relative motion of ions and electrons becomes significant.
Originating in plasma physics \cite{AchesonHide73,Lighthill60}, the system is relevant to a variety of physical phenomena,
including magnetic reconnection and the dynamics of protostellar
disks \cite{BalbusTerquem01,Forbes91}. Its main difference from the classical magnetohydrodynamic (MHD) system is the presence
of the Hall term
\[
\curl\bigl((\curl b)\times b\bigr),
\]
which is quadratic in the magnetic field $b$ and involves two spatial derivatives. Consequently, this term has the same
differential order as the Laplacian, while its nonlinear structure introduces additional difficulties in the presence of fractional
dissipation. These features make the analysis of Hall-MHD
substantially more delicate than that of the classical MHD system.
We refer to \cite{CDL14,chae-schonbek} for related results on
the well-posedness and evolution of Hall-MHD equations.

In this paper, we address a Liouville-type problem for the stationary fractional Hall-MHD system in 
 $\mathbb R^3$. More precisely, we consider 
\begin{equation}\label{SHMHD}
\begin{cases}
(-\Delta)^{\alpha/2}u+(u\cdot\nabla)u+\nabla\Pi
=(b\cdot\nabla)b,\\[2mm]
(-\Delta)^{\beta/2}b+(u\cdot\nabla)b-(b\cdot\nabla)u
+\curl\bigl((\curl b)\times b\bigr)=0,\\[2mm]
\nabla\cdot u=0,\qquad \nabla\cdot b=0,
\end{cases}
\qquad 0<\alpha,\beta<2,
\end{equation}
subject to the decay conditions
\[
u(x) \to 0, 
\quad 
b(x) \to 0, 
\quad 
\text{as } 
|x|\to \infty
.
\]
Here, $u$ and $b$ denote the velocity and magnetic fields, respectively, while $\Pi$ is the total pressure. The fractional Laplacians $(-\Delta)^{\alpha/2}$ and $(-\Delta)^{\beta/2} $ are defined by the Fourier multipliers with symbols $|\xi|^\alpha$ and $|\xi|^\beta $, respectively. Throughout the paper, the coefficient of the Hall term is normalized to one, see Remark \ref{rem:coef} below. In the absence of the Hall term, system \eqref{SHMHD} reduces to the stationary fractional MHD system.

 Our goal is to identify conditions on the dissipation exponents and on
the velocity and magnetic fields under which every smooth
stationary solution vanishing at infinity must be trivial.
Liouville theorems of this kind are  closely related to the
possibility of ruling out nontrivial stationary states under suitable energy assumptions. Even for the stationary
Navier--Stokes equations in $\mathbb R^3$,  the question remains  open under the sole assumption of finite
 Dirichlet integral.
Known results therefore impose additional conditions on the solution. A classical example is the integrability 
assumptionn
$u\in L^{9/2}(\mathbb R^3)$ \cite{Galdi}, 
while other approaches
involve suitable assumptions on oscillation, head pressure, or
Morrey-space norms
\cite{chae14,CW,Kozonoetal,Seregin16,Sereginetwang,V2026}. Despite their different formulations, many of these arguments
rely on a common principle: testing the equations against localized versions of the solution yields 
energy estimates in
which the dissipation is balanced by terms supported in an annulus, and the desired conclusion follows once these terms
vanish as the annulu moves to infinity.

For the stationary Hall-MHD system with  classical Laplacian
dissipation, this strategy has been developed in
\cite{chaewolfMHD,chaekimwolf}, where  the Hall term is incorporated  into the energy estimates alongside the other nonlinear terms.
 Subsequent works have established  Liouville-type results under various assumptions on the behavior
of the velocity and magnetic fields   at infinity, including integrability and growth conditions in Lebesgue-type spaces
\cite{ChenLiWang22,ChoNeustupaYang24,Zhang25}, as well as conditions allowing the fields to approach a nonzero constant
state \cite{WangYang25}. A useful structural property of the Hall
term is the identity
\[
\bigl((\curl b)
\times b\bigr)\cdot\curl  
b=0
,
\]
which shows that it does not contribute 
directly to the
dissipation in the formal energy identity. Nevertheless, after
localization, the Hall term produces a
 boundary contribution
that must be estimated. The aforementioned results concern
the classical Laplacian, corresponding to $\alpha=\beta=2$
in the notation adopted below.

Replacing the Laplacian by 
fractional powers introduces a further difficulty. 
The operators $ ( -\Delta )^{\alpha/2}$ and
$(-\Delta )^{ \beta/2} $ are nonlocal, so the localization arguments used in the classical setting   cannot be transferred directly.
For the fractional Navier--Stokes equations,  several approaches have been developed to address this issue
\cite{EVH2026,J-me,LeeLee26,LiuTan26,LiuZuo,wx18,yang22,zeng}.
Among them, the Caffarelli--Silvestre extension    \cite{cs07}  is particularly useful, 
 it represents a fractional
Laplacian as a boundary operator associated with
 a local, weighted elliptic problem in the upper half-space.

For the fractional MHD system without the Hall term, the region
\[
\mathcal A
=
\Bigl\{
(\alpha,\beta)\in\mathbb R^2
\ \Big|\
1\le\alpha\le\tfrac53,\quad
0<
\beta<2,
\quad
3\le\alpha
+
2\beta\le5
\Bigr\}
\]
was identified in \cite{AV26} as a range of dissipation exponents for which the Liouville property holds under
finite-energy assumptions. The corresponding problem for
fractional Hall-MHD is less developed. To the best of our knowledge, the only previously available Liouville theorem for this system is due to \cite{WL26}. By combining the Caffarelli--Silvestre extension with a fractional Leibniz rule,
those authors proved triviality for
$1 \le \alpha <2$ and $\frac53 \le \beta <2$, under finite fractional energy and the additional assumption
$u,b \in L^p(\mathbb R^3) $ for some
$2\le p\le\frac92$.

These results lead to a natural question: can the finite-energy criteria established for fractional MHD in
\cite{AV26} be retained when the Hall term is present, without
imposing additional integrability assumptions? The main obstruction is the boundary contribution generated by the Hall
term. In the physical variables, this contribution involves a full spatial derivative of $b$, whereas the magnetic dissipation
controls only $\|(-\Delta)^{\beta/4}b\|_{L^2}$. Our approach resolves this mismatch by expressing the Hall contribution
through the Caffarelli--Silvestre extension of the magnetic field and estimating it in terms of the associated weighted
energy.
 The key point is that, when $ \beta \ge 1$, the resulting
estimate can be obtained using the boundedness of $ b$,  
which
follows from its smoothness and decay at infinity, together
with its finite fractional energy. 
Thus,  no additional integrability assumption on the fields is needed.

We show that, throughout the portion of the parameter region
$\mathcal A $ where $\beta \ge 1$, the Liouville property for fractional MHD extends to the fractional Hall-MHD system.
More precisely,  define
\[
\mathcal A_H
=
\Bigl\{
(\alpha,\beta):
1\le\alpha\le\tfrac53,\quad
1\le\beta<2,\quad
\alpha+2\beta\le5
\Bigr\}.
\]
Our main result is the following theorem.

\begin{theorem}\label{thm1}
Let $(u, b) $ be a smooth solution   of   \eqref{SHMHD} such that $u(x) , b(x) \to 0$ as $|x|\to\infty$ and $u\in\dot H^{\alpha/2}(\mathbb R^3),$ $b\in\dot H^{\beta/2}(\mathbb R^3)$, with $(\alpha,\beta) \in \mathcal A_H$. Then $ u \equiv0, \ b\equiv 0$.
\end{theorem}

We now discuss several aspects of Theorem~\ref{thm1}, with emphasis on the range of admissible parameters and the influence of the Hall term.

\begin{remark}\label{rem:range}
The region $\mathcal A_H$ can be viewed as the part of the
admissible region for fractional MHD in \cite{AV26} that lies
above the threshold $\beta=1$; that is,
\[
\mathcal A_H=\mathcal A\cap\{\beta\ge1\}.
\]
In particular, the parameter range obtained here contains the ranges covered by the fractional MHD results of Zeng
\cite{zeng} and Lee--Lee \cite{LeeLee26}, namely
\[
1\le\alpha,\beta\le\frac53
\]
and
\[
1\le\alpha<\frac53,\qquad
\frac53<\beta<2,\qquad
\alpha+2\beta<5,
\]
respectively. Hence, in these regimes, the finite-energy Liouville criteria remain valid after the Hall term is
introduced, without requiring additional integrability assumptions. Figure~\ref{fig:AH} illustrates the region
$\mathcal A_H$.

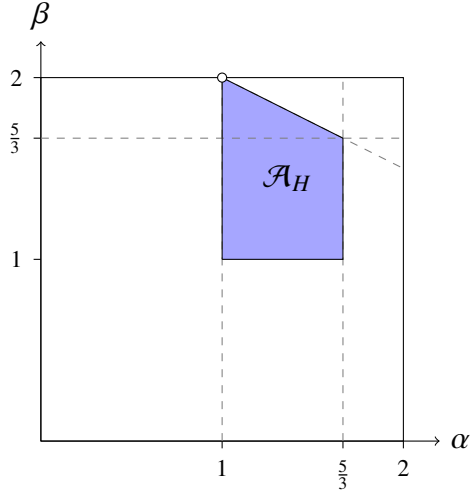
\begin{figure}[htbp]
\centering
\begin{tikzpicture}[scale=2.4]
\fill[blue!35] (1,1)--({5/3},1)--({5/3},{5/3})--(1,2)--cycle;
\axesframe
\guides
\draw (1,1)--({5/3},1)--({5/3},{5/3})--(1,2);
\draw (1,1)--(1,2);
\draw[fill=white] (1,2) circle (0.025);
\node at (1.36,1.45) {$\mathcal A_H$};
\end{tikzpicture}
\caption{The region $\mathcal A_H $ in the window $ (0 ,2)^2 $. The point $ (1,2)$ is excluded because $\beta < 2$. The gray dashed lines are the reference lines $\alpha =1 $, $\alpha=\frac53$, $\beta = \frac53$,   and $\alpha+2\beta=5$.}
\label{fig:AH}
\end{figure}
\end{remark}

\begin{remark}\label{rem:Hall}
The restriction $\beta\ge1$ arises from the estimate of the Hall contribution rather than from the energy estimate for the
remaining terms. Indeed, testing the magnetic equation with  $\psi_R b $, where $\psi_R$ is a smooth cut-off function, 
and using
\[
\bigl((\curl b)\times b\bigr)\cdot\curl b
=
0
,
\]
shows that the Hall term contributes only through the boundary
flux
\[
H_R
=
\int_{\mathbb R^3}
\Bigl[
(\curl b\cdot\nabla\psi_R)|b|^2
-
(\curl b\cdot b)(b\cdot\nabla\psi_R)
\Bigr]\,dx,
\]
as established in Lemma~\ref{lemma.Hall.id}. The difficulty is that this expression contains a full derivative of $b$ in
the physical variables, whereas the available dissipation controls only the fractional derivative of order $\beta/2$.
To overcome this difficulty, we express $H_R$ in terms of the
Caffarelli--Silvestre extension $ b^* $ and estimate it using
the weighted energy 
\[
\int_{\mathbb R^4_+}
y^{1-\beta}|\bar\nabla b^*|^2\,
dx\,dy
.
\]
When $\beta\ge1$, the exponent $1-\beta$ is nonpositive. Consequently, on the support of the cut-off in the extended
variable, the unweighted gradient can be controlled by the
weighted energy, up to a factor depending on $R$. Together with the boundedness of $b$, this yields the decay
of $H_R$ as $R\to\infty$. This estimate is the essential ingredient that allows the energy argument to accommodate
the Hall term under the assumptions of Theorem~\ref{thm1}.
\end{remark}

\begin{remark}\label{rem:beta<1}
The argument above does not cover the part of the MHD parameter region $\mathcal A$ where $\beta<1$. In this
regime, the weight $y^{1-\beta}$ has a positive exponent and degenerates as $y\downarrow0$. As a result, the weighted
energy no longer provides the control of the unweighted quantity $|\bar\nabla b^*|^2$ near the boundary
$\{y=0\}$ that is needed in the estimate of the Hall flux. More specifically, the term $J_2$ in the proof of
Lemma~\ref{lemma.Hall.decay} cannot be controlled by the available weighted energy using our argument. Whether the
Liouville property remains valid in this region under the finite-energy assumptions alone is an open question.
\end{remark}

\begin{remark}\label{rem:coef}
The normalization of the Hall coefficient to one in \eqref{SHMHD} does not restrict the result. Indeed, if the
Hall term is multiplied by a constant $\kappa\in\mathbb R$, the same proof applies, since $\kappa$ enters the energy
identity only as a multiplicative factor in $H_R$. For $\kappa=0$, the system reduces to the fractional MHD system,
and the corresponding Liouville property follows from
\cite{AV26}.
\end{remark}

\begin{remark}\label{rem:NS}
The theorem also recovers the corresponding Liouville result for the stationary fractional Navier--Stokes equations.
Indeed, if $b\equiv0$, the magnetic equation is satisfied identically and the momentum equation reduces to the
stationary fractional Navier--Stokes system. In this case, the range $1\le\alpha\le\frac53$ from Corollary~1.2 of
\cite{AV26} is recovered, in agreement with \cite{zeng} and \cite{LeeLee26}.
\end{remark}

The rest of the paper is organized as follows.
In Section~\ref{sect:prel}, we recall the
Caffarelli--Silvestre extension and establish properties of
the extensions of bounded smooth functions that are needed
to handle the Hall contribution. Section~\ref{proofs} is
devoted to the proof of Theorem~\ref{thm1}. \\

Given $R>1$, in what follows, and throughout this paper, we use the notation
\[
\mathcal{C}_R := \{ x \in \mathbb{R}^3 \ : \ R < |x| < 2R \}.
\]
We let $\psi\in \mathcal{C}^\infty_c(\mathbb{R}^3)$ be a smooth radial non-increasing cut-off
function such that $0\le  \psi\le  1$, $\psi(x)=1$ for $|x|\le  1$, and
$\psi(x)=0$ for $|x|\ge 2$. Given $R>1$, set
$  \psi_R(x)=\psi\!\left(\frac{x}{R}\right)$. On the other hand,  we  consider as well  a non-increasing function
$\chi\in \mathcal{C}^\infty([0,\infty))$ such that $\chi(y)=1$ for
$y\leq 1$ and $\chi(y)=0$ for $y\geq 2$. Given $R>1$, we define
$\chi_R(y) = \chi\!\left(\frac{y}{R}\right)$.  

\section{Preliminaries}\label{sect:prel}

The nonlocal character of the fractional operators $(-\Delta)^{\alpha/2}$ and $(-\Delta)^{\beta/2}$ prevents the direct application of several standard techniques employed in Liouville-type results for the stationary Navier--Stokes and MHD systems.  
To  overcome this difficulty, we adopt the Caffarelli--Silvestre extension framework, which realizes a  fractional Laplacian in $\mathbb{R}^{3}$ through a local elliptic problem in the upper half-space $\mathbb{R}^{4}_{+}$, governed by the extended Laplacian
 $ \bar{\Delta}$. Throughout the paper, the symbols $\bar{\nabla}$, $\overline{\operatorname{div}}$, and $\bar{\Delta}$ denote the gradient, divergence, and Laplacian, respectively, with respect to the extended variables. Because the two equations
 in \eqref{SHMHD} contain fractional Laplacians of possibly different orders, separate extensions must be introduced for $u$ and $b$, each associated with its corresponding weight. We therefore formulate the following result for an arbitrary exponent $\sigma$, so that it can be applied to both extensions.

\begin{theorem}\cite{cs07}\label{thm:CS}
Given $\sigma \in (0,2)$,
consider $v \in \dot{H}^{\frac \sigma 2}(\mathbb{R}^3)$, $\mu = 1 - \sigma$,
and the weighted homogeneous energy space
$\dot H^1(\mathbb R^4_+,y^\mu)$, obtained as the completion of
$C_c^\infty(\overline{\mathbb{R}^4_+})$ with respect to the seminorm
\[
\|w\|_{\dot H^1(\mathbb R^4_+,y^\mu)}
:=
\left(
\int_{\mathbb R^4_+}
y^\mu |\bar\nabla w|^2\, d  x\, d  y
\right)^{1/2}.
\]
 Then $v$ admits a unique extension $v^* \in \dot H^1(\mathbb{R}^4_+, y^\mu)$ satisfying
\[
\bar{\Delta}_\mu v^*(x, y) := \bar{\Delta} v^* + \frac{\mu}{y} \partial_y v^* = \frac{1}{y^\mu} \overline{\operatorname{div}}(y^\mu \bar{\nabla} v^*) = 0  ,
\]
and
\[
v^*(x, 0) = v(x).
\]
In addition, there exists $C_\sigma>0$, depending only on $\sigma$,  such that:
\begin{enumerate}
    \item the operator $(-\Delta)^{\frac \sigma 2}  v$ is given by
    \[
    (-\Delta)^{\frac \sigma 2}  v(x) = - \lim_{y \to 0} C_\sigma ( y^\mu \partial_y v^*(x, y)),
    \]
    \item and, the energy equality
    \[
    \int_{\mathbb{R}^3} |(-\Delta)^{\frac \sigma 4}  v|^2 dx  = C_\sigma \int_{\mathbb{R}^4_+} y^\mu |\bar{\nabla} v^*|^2 dx dy.
    \]
\end{enumerate}
Moreover, $v^*(\cdot,y)=P^{(\sigma)}_y*v$, where $P_y^{(\sigma)}(x)=c_{3,\sigma}\,y^{\sigma}(|x|^2+y^2)^{-\frac{3+\sigma}{2}}$ is the fractional Poisson kernel and $c_{3,\sigma}>0$ is chosen so that $\|P_y^{(\sigma)}\|_{L^1(\mathbb R^3)}=1$.
\end{theorem}

In what follows we write
\[
\mu:=1-\alpha, \qquad \nu:=1-\beta, \qquad p:=\frac{6}{3-\alpha}, \qquad q:=\frac{6}{3-\beta},
\]
so that, by the Sobolev embedding, $u\in L^{p}(\R^3)$ and $b\in L^{q}(\R^3)$. We denote by $u^*$ and $b^*$ the Caffarelli--Silvestre extensions of $u$ and $b$, associated with the weights $y^{\mu}$ and $y^{\nu}$ respectively, and by $C_\alpha$, $C_\beta$ the corresponding constants given by Theorem \ref{thm:CS}. Note that $\mu,\nu>-1$ because $\alpha,\beta<2$, and that $\nu\le0$ if and only if $\beta\ge1$.

\section{Proof of Theorem \ref{thm1}}\label{proofs}

We begin with the lemmas that will be used in the proof of Theorem \ref{thm1}. 
The first one is stated for a generic exponent, so that it can be applied both to $u$ and to $b$. 
Its proof is analogous to that of \cite{EVH2026}, and is  omitted for brevity.
The proofs of all the subsequent lemmas are postponed to the end of this section.

\begin{lemma}\label{lemma.with.E}
Let $\sigma\in(0,2)$, $v\in \dot H^{\sigma/2}(\mathbb R^3)$ and let $v^*$ be its Caffarelli--Silvestre extension. For every $R>1$ let $E_R\subseteq B_{2R}$ be a measurable set. Then
\[
\frac1{R^2}\int_0^{2R}\int_{E_R} y^{1-\sigma}|v^*|^2 dxdy\xrightarrow[R\to\infty]{}0
.
\]
\end{lemma}

\begin{lemma}\label{lemma I R}
Let $1\le\alpha\le\frac53$ and $3\le\alpha+2\beta\le5$, $u\in\dot H^{\alpha/2}(\R^3)$, $b\in\dot H^{\beta/2}(\R^3)$. Let $\Pi$ be the pressure defined in \eqref{eq:pressure} below.
 Then
\[
\lim_{R\to +\infty}
\int_{ C_R }\bigl(|\vu|^2+|b|^2\bigr)
| \vu |
| \nabla  \psi_R |
 dx
+
 \int_{ C_R }
 | \Pi  |
 | \vu  |
|  \nabla  \psi_R |
  dx
 = 0.
\]
\end{lemma}

\begin{lemma}\label{lemma.Hall.id}
Let $b\in C^\infty(\mathbb R^3)$ and $R>1$. Then
\begin{multline*}
\int_{\mathbb R^3}\curl\bigl((\curl b)\times b\bigr)\cdot b\,\psi_R\,dx
\\
=
\int_{\mathbb R^3}\Bigl[(\curl b\cdot\nabla\psi_R)\,|b|^2-(\curl b\cdot b)\,(b\cdot\nabla\psi_R)\Bigr]dx=:H_R .
\end{multline*}
\end{lemma}

\begin{lemma}\label{lemma.Hall.decay}
Let $1\le\beta<2$ and let $b\in C^\infty(\R^3)\cap L^\infty(\R^3)$ with $b\in\dot H^{\beta/2}(\R^3)$. Set
\[
M:=\|b\|_{L^\infty(\R^3)},\qquad E:=\int_{\mathbb R^4_+}y^{1-\beta}|\bar\nabla b^*|^2\,dx\,dy ,
\]
where $b^*$ is the Caffarelli--Silvestre extension of $b$. Then $E<\infty$ and there exists $C=C(\beta)>0$ such that, for every $R>1$,
\[
|H_R|\le  C\,M\,\bigl(E+\|b\|_{L^q(\R^3)}E^{1/2}\bigr)\,R^{\beta-2},\qquad q=\frac{6}{3-\beta}.
\]
In particular, $H_R\to0$ as $R\to\infty$.
\end{lemma}

\begin{proof}[Proof of Theorem \ref{thm1}]

To begin, note that, since $u\in \dot H^{\alpha/2}(\mathbb{R}^3)$ and $b\in \dot H^{\beta/2}(\mathbb{R}^3)$, the Sobolev
embedding gives $u\in L^{p}(\mathbb{R}^3)$ and $b\in L^{q}(\mathbb{R}^3)$. Since $\alpha,\beta>0$ we have $p,q>2$, and therefore $u_iu_j\in L^{p/2}(\mathbb{R}^3)$ and $b_ib_j\in L^{q/2}(\mathbb{R}^3)$ with $1<\frac p2,\frac q2<+\infty$. Then, taking the divergence of the momentum equation and using
$\nabla\cdot u=\nabla\cdot b=0$, we obtain the elliptic equation
\begin{equation}\label{ellipeq}
-\Delta \Pi=\partial^i\partial^j(u_iu_j-b_ib_j)=:f .
\end{equation}
Note that the Hall term does not appear in this equation, because it only enters the induction equation.
Now, since \eqref{ellipeq} determines $\Pi$ only up to an additive harmonic function, we fix the
pressure through the Calder\'on--Zygmund representation
\begin{equation}\label{eq:pressure}
\Pi:=\Pi_1+\Pi_2, \qquad \Pi_1:=R^iR^j(u_iu_j), \qquad \Pi_2:=-R^iR^j(b_ib_j),
\end{equation}
where $R^i$ denotes the $i$-th Riesz transform on $\mathbb{R}^3$. By the Calder\'on--Zygmund
theorem,
\begin{equation}\label{eq:CZ}
\|\Pi_1\|_{L^{\frac p2}(\mathbb{R}^3)}\le  C\,\|u\|_{L^{p}(\mathbb{R}^3)}^2<+\infty,
\qquad
\|\Pi_2\|_{L^{\frac q2}(\mathbb{R}^3)}\le  C\,\|b\|_{L^{q}(\mathbb{R}^3)}^2<+\infty .
\end{equation}
We claim that this is the only solution   of \eqref{ellipeq} in $L^{\frac p2}(\mathbb{R}^3)+L^{\frac q2}(\mathbb{R}^3)$. Indeed, if
$\tilde \Pi$ is another solution in this space, then $h:=\Pi-\tilde \Pi=h_1+h_2$ is harmonic on
$\mathbb{R}^3$, with $h_1\in L^{\frac p2}(\mathbb{R}^3)$ and $h_2\in L^{\frac q2}(\mathbb{R}^3)$. Thus, considering the mean value property and H\"older's
inequality, for every $x\in\mathbb{R}^3$ and $\rho>0$,
\[
|h(x)|\le  \frac{1}{|B_\rho|}\int_{B_\rho(x)}(|h_1|+|h_2|)\,dy
\le  |B_\rho|^{-\frac{2}{p}}\,\|h_1\|_{L^{\frac p2}(\mathbb{R}^3)}
+|B_\rho|^{-\frac{2}{q}}\,\|h_2\|_{L^{\frac q2}(\mathbb{R}^3)}\xrightarrow[\rho\to\infty]{}0 ,
\]
thus, we get $h\equiv 0$. Throughout the paper, $\Pi$ always denotes this normalized pressure.
On the other hand,  since  $u$ and $b$ are smooth, by using elliptic regularity theory on \eqref{ellipeq}, we conclude that the pressure term $\Pi$ is also smooth.

To continue,  we test the first equation of  \eqref{SHMHD}  with $\psi_R u $ and the second one with $\psi_R b$. We denote
\[
D_u:=\int_{\mathbb{R}^3} \psi_R u \cdot (-\Delta)^{\frac \alpha 2}  u \, dx,
\qquad
D_b:=\int_{\mathbb{R}^3} \psi_R b \cdot (-\Delta)^{\frac \beta 2}  b \, dx ,
\]
and we compute each of the remaining terms, using $\nabla\cdot u=\nabla\cdot b=0$, integrating by parts (all the integrands have compact support, since $\psi_R\in C_c^\infty$) and   Lemma \ref{lemma.Hall.id} we obtain
\begin{equation}\label{eq:energy.Hall}
D_u+D_b+H_R=\int_{\mathbb{R}^3} \left(\frac {|u|^2} 2 +\frac {|b|^2} 2 + \Pi \right)u \cdot \nabla \psi_R \, dx
-\int_{\mathbb{R}^3} (u\cdot b)\,(b \cdot \nabla \psi_R) \, dx.
\end{equation}
 
  We next treat the fractional dissipation terms by means of the Caffarelli--Silvestre extension. For this purpose, we consider $\Phi(x,y):=\psi_R(x)\chi_R(y)$ and the identity
(see \cite[Eq. (3.2)]{yang22}), applied to each component of $u$ with $\sigma=\alpha$ and to each component of $b$ with $\sigma=\beta$,
\begin{align}
D_u
& =
 C_\alpha \int_{\mathbb{R}^4_+} y^\mu |\overline\nabla u^*|^2 \Phi \, dx \, dy
 -
  C_\alpha \int_{\mathbb{R}^4_+} \frac{|u^*|^2}{2}
   \overline
 {\text{div}}
  \bigl( y^\mu \overline\nabla \Phi \bigr) \, dx \, dy
  ,\label{eq25-u}\\
D_b
& =
 C_\beta \int_{\mathbb{R}^4_+} y^\nu |\overline\nabla b^*|^2 \Phi \, dx \, dy
 -
  C_\beta \int_{\mathbb{R}^4_+} \frac{|b^*|^2}{2}
   \overline
 {\text{div}}
  \bigl( y^\nu \overline\nabla \Phi \bigr) \, dx \, dy
  .\label{eq25-b}
\end{align}
Then plugging \eqref{eq25-u} and \eqref{eq25-b} into \eqref{eq:energy.Hall}, we write
\begin{multline}\label{eq3-energy}
 C_\alpha\int_{\mathbb{R}^4_+} y^\mu |\overline \nabla u^*|^2 \Phi \, dx \, dy
 +
 C_\beta\int_{\mathbb{R}^4_+} y^\nu |\overline \nabla b^*|^2 \Phi \, dx \, dy
\\=
  \frac{C_\alpha}{2} \int_{\mathbb{R}^4_+}  |u^*|^2   \overline
 {\text{div}} \bigl( y^\mu  \overline\nabla \Phi \bigr) \, dx \, dy
 +
  \frac{C_\beta}{2} \int_{\mathbb{R}^4_+}  |b^*|^2   \overline
 {\text{div}} \bigl( y^\nu  \overline\nabla \Phi \bigr) \, dx \, dy
\\
+
 \int_{\mathbb{R}^3} \left(\frac {|u|^2} 2 +\frac {|b|^2} 2 + \Pi \right)u \cdot \nabla \psi_R \, dx
-\int_{\mathbb{R}^3} (u\cdot b)\,(b \cdot \nabla \psi_R) \, dx
-H_R.
\end{multline}

Now, considering the supports of $\chi_R,\ \psi_R$ and the facts that $\overline\nabla \Phi(x,y) = ( \chi_R \nabla\psi_R, \psi_R \chi'_R)$ and, for a generic weight $y^{\kappa}$, we can write 
\[ \overline
 {\text{div}} ( y^\kappa \overline\nabla \Phi ) = y^\kappa \chi_R\Delta \psi_R + \psi_R\bigl(\kappa y^{\kappa-1}\chi_R'+y^\kappa\chi_R''\bigr),\] 
 we notice that the first term is supported in $\mathcal C_R\times[0,2R]$, while the second one is supported in $B_{2R}\times[R,2R]$, where $y^{\kappa-1}\le  R^{-1}y^{\kappa}$. We deduce the estimate
 
\[
\Bigl|\int_{\mathbb{R}^4_+}  |v^*|^2  \overline
 {\text{div}} \bigl( y^\kappa \overline\nabla \Phi(x,y)  \bigr) dx  dy\Bigr|
\leq
\frac C {R^2}
 \int_{0}^{2R}   \!  \!   \!
 \int_{ \mathcal C_R} y^\kappa |v^*|^2 dxdy
 +
 \frac C {R^2}
 \int_{R}^{2R} \!  \!  \!     \int_{ B_{2R}} y^\kappa |v^*|^2 dx dy ,
\]
where $C$ depends on $\kappa$ but not on $R$. Take $(v,\kappa)=(u,\mu)$ and $(v,\kappa)=(b,\nu)$. Then, applying Lemma \ref{lemma.with.E} (with $\sigma=\alpha$, respectively $\sigma=\beta$) with $E_R=\mathcal C_R$, we conclude that the first term on the right-hand side tends to zero as $R\to\infty$. On the other hand, since $y^\kappa|v^*|^2 \geq 0$, we can write
\[
 \int_{R}^{2R} \!  \!  \!     \int_{ B_{2R}} y^\kappa |v^*|^2 dx dy
\le
 \int_{0}^{2R} \!  \!  \!     \int_{ B_{2R}} y^\kappa |v^*|^2 dx dy,
 \]
 therefore, considering Lemma \ref{lemma.with.E} with $E_R=B_{2R}$ we obtain that the second term tends to zero also. Considering the same approach for $b^*$, we conclude
 \begin{equation}
 \lim_{R\to\infty}
 \int_{\mathbb{R}^4_+}  |u^*|^2 \overline{\text{div}} \bigl( y^\mu \overline\nabla \Phi(x,y)  \bigr) dx  dy
 =
 \lim_{R\to\infty}
 \int_{\mathbb{R}^4_+}  |b^*|^2 \overline{\text{div}} \bigl( y^\nu \overline\nabla \Phi(x,y)  \bigr) dx  dy
 =0.
 \end{equation}

Now, since $\psi_R$ is radial and $\nabla \psi_R$ is supported in $\mathcal C_R$, we can write
\begin{multline}
\Bigl|\int_{\mathbb{R}^3} \left(\frac {|u|^2} 2 +\frac {|b|^2} 2 + \Pi \right)u \cdot \nabla \psi_R \, dx
-\int_{\mathbb{R}^3} (u\cdot b)\,(b \cdot \nabla \psi_R) \, dx\Bigr|
\\ \leq
C
\int_{ \mathcal C_R }
\bigl(|\vu|^2+|b|^2\bigr)
| \vu |
| \nabla   \psi_R|
 dx
+
C
 \int_{ \mathcal C_R }
 | \Pi  |
 | \vu |
|  \nabla \psi_R |
  dx
  ,
\end{multline}
where we used $|(u\cdot b)(b\cdot\nabla\psi_R)|\le|u||b|^2|\nabla\psi_R|$. Since $(\alpha,\beta)\in\mathcal A_H\subset\mathcal A$, Lemma \ref{lemma I R} shows that the right-hand side tends to zero as $R\to\infty$.

Since $u$ and $b$ are smooth and $b(x)\to0$ as $|x|\to\infty$, the field $b$ is bounded, and $b\in\dot H^{\beta/2}(\R^3)$ with $1\le\beta<2$ because $(\alpha,\beta)\in\mathcal A_H$. Hence, Lemma \ref{lemma.Hall.decay} gives $H_R\to0$ as $R\to\infty$. This is the only place where the Hall term, the decay of $b$, and the condition $\beta\ge1$ are used.

 Thus, provided with these results,  we take the limit as $R\to \infty$ in \eqref{eq3-energy}. Since $\psi$ is radial and non-increasing, and $\chi$ is non-increasing, $\Phi=\Phi_R$ increases pointwise to $1$ as $R\to\infty$, and the monotone convergence theorem applied to the left-hand side of \eqref{eq3-energy} yields
\[
C_\alpha\int_{\mathbb{R}^4_+} y^\mu |\overline{\nabla} u^*|^2 \, dx \, dy
+
C_\beta\int_{\mathbb{R}^4_+} y^\nu |\overline{\nabla} b^*|^2 \, dx \, dy = 0.
\]
Thus, since $C_\alpha,C_\beta>0$, the both integrals above vanish. 
Hence, we conclude that  $u^*$ and $b^*$ are constant in $\mathbb R^4_+$, and so are their traces $u(x)=u^*(x,0)$ and $b(x)=b^*(x,0)$. 
Then, considering that $u\in L^{p}(\mathbb R^3)$ and $b\in L^{q}(\mathbb R^3)$, and the fact that the only constant function belonging to
 $ L^{r}(\mathbb R^3) $, $r<\infty$, is the zero function, we conclude that  $u\equiv 0$ and $b\equiv 0$. This finishes the proof.
\end{proof}
 
\begin{proof}[Proof of Lemma \ref{lemma I R}]
Set
\[
p=\frac{6}{3-\alpha},
\qquad
q=\frac{6}{3-\beta}.
\]
The assumptions $u\in\dot H^{\alpha/2}(\R^3)$ and
$b\in\dot H^{\beta/2}(\R^3)$, together with the Sobolev embeddings, yield
\[
u\in L^p(\R^3),\qquad b\in L^q(\R^3).
\]
We shall repeatedly use the following estimate. Suppose that
$\rho_1,\rho_2,\rho_3\in[1,\infty]$ satisfy
\[
\vartheta=1-\frac1{\rho_1}-\frac1{\rho_2}-\frac1{\rho_3}
\in[0,1).
\]
H\"older's inequality, together with the bounds
$\|\nabla\psi_R\|_{L^\infty}\le  CR^{-1}$ and
$|\mathcal C_R|\le  CR^3$, gives
\begin{align}\label{eq:scaling}
\int_{\mathcal C_R}|f_1||f_2||f_3||\nabla\psi_R|\,dx
&\le
\prod_{j=1}^3\|f_j\|_{L^{\rho_j}(\mathcal C_R)}
\|\nabla\psi_R\|_{L^{1/\vartheta}(\mathcal C_R)}
\nonumber\\
&\le
CR^{3\vartheta-1}
\prod_{j=1}^3\|f_j\|_{L^{\rho_j}(\mathcal C_R)}.
\end{align}
In the case when $\vartheta=0$, the norm of $\nabla\psi_R$ is interpreted as its
$L^\infty$ norm. 
In particular, whenever $3\vartheta-1\le  0 $, the
factor $R^{3\vartheta-1}$ 
is bounded independently of $R>1$.
Now, for convenience, we write
\[
\begin{aligned}
I_1(R)
&=\int_{\mathcal C_R}
|u|^3| \nabla\psi_R|\,dx, 
\\
I_2(R)&
=\int_{\mathcal C_R}|b|^2|u||\nabla\psi_R|\,dx,\\
I_3(R)&
= 
\int_{\mathcal C_R}|\Pi||u||\nabla\psi_R|\,dx.
\end{aligned}
\]
Thus, it suffices to show that each of these quantities tends to zero as
$R\to\infty$.

For $I_1(R)$, choose $\rho_1=\rho_2=\rho_3=p$ in
\eqref{eq:scaling}. The, the corresponding exponent is
\[
\vartheta_1=1-\frac3p=\frac{\alpha-1}{2},
\qquad
3\vartheta_1-1=\frac{3\alpha-5}{2}\le0.
\]
Consequently, we can write 
\[
I_1(R)\le  C\|u\|_{L^p(\mathcal C_R)}^3.
\]
Since $u \in L^p(\R^3)$ and  the 
 annuli $\mathcal C_R$ escape to infinity, 
we get 
\[
\lim_{R\to\infty}\|u\|_{L^p(\mathcal C_R)}=0.
\]
Then, it follows that $ I_1(R) \to0 $.

Next, applying \eqref{eq:scaling} with
$(\rho_1, \rho_2,\rho_3) = (p ,q,q )$, we find that 
\[
\vartheta_2
=1-\frac1p-\frac2q
=\frac{\alpha+2\beta-3}{6}.
\]
Note that, the parameter restrictions $3\le\alpha+2\beta\le5$ ensure that
$\vartheta_2\ge0$ and
\[
3\vartheta_2-1=\frac{\alpha+2\beta-5}{2}\le0.
\]
Therefore, we can write 
\[
I_2(R)\le
C\|u\|_{L^p(\mathcal C_R)}
\|b\|_{L^q(\mathcal C_R)}^2
\le
C\|u\|_{L^p(\mathcal C_R)}\|b\|_{L^q(\R^3)}^2.
\]
The first factor tends to zero, while the second is finite. Hence
$I_2(R)\to0$ as $R\to\infty$.

It remains to deal with the pressure contribution. Decompose
$\Pi=\Pi_1+\Pi_2$ as in \eqref{eq:pressure}. For the first component,
the Calder\'on--Zygmund estimate \eqref{eq:CZ} gives
\[
\|\Pi_1\|_{L^{p/2}(\R^3)}
\le  C\|u\|_{L^p(\R^3)}^2.
\]
Using \eqref{eq:scaling} with
$(\rho_1,\rho_2,\rho_3)=(p/2,p,\infty)$ and $f_3\equiv1$, we obtain
\[
\begin{aligned}
\int_{\mathcal C_R}|\Pi_1||u||\nabla\psi_R|\,dx
&\le
CR^{3\vartheta_1-1}
\|\Pi_1\|_{L^{p/2}(\mathcal C_R)}
\|u\|_{L^p(\mathcal C_R)}\\
&\le
C\|u\|_{L^p(\R^3)}^2
\|u\|_{L^p(\mathcal C_R)}.
\end{aligned}
\]
Here we used $3\vartheta_1-1\le0$ and
$\|1\|_{L^\infty(\mathcal C_R)}=1$. The right-hand side vanishes as
$R\to\infty$.

Similarly, for the magnetic contribution, \eqref{eq:CZ} yields
\[
\|\Pi_2\|_{L^{q/2}(\R^3)}
\le  C\|b\|_{L^q(\R^3)}^2.
\]
Thus, choosing $(\rho_1,\rho_2,\rho_3)=(q/2,p,\infty)$ in
\eqref{eq:scaling}, with the first factor representing $\Pi_2$, gives
\[
\begin{aligned}
\int_{\mathcal C_R}|\Pi_2||u||\nabla\psi_R|\,dx
&\le
CR^{3\vartheta_2-1}
\|\Pi_2\|_{L^{q/2}(\mathcal C_R)}
\|u\|_{L^p(\mathcal C_R)}\\
&\le
C\|b\|_{L^q(\R^3)}^2
\|u\|_{L^p(\mathcal C_R)}
\longrightarrow0
\end{aligned}
\]
as $R\to\infty$. Thus both pressure contributions vanish, and
therefore $I_3(R)\to0$.

Combining the three limits, we conclude that
\[
\lim_{R\to\infty}I_1(R)
=
\lim_{R\to\infty}I_2(R)
=
\lim_{R\to\infty}I_3(R)=0,
\]
which proves the lemma.
\end{proof}

\begin{proof}[Proof of Lemma \ref{lemma.Hall.id}]
Let $F:=(\curl b)\times b$. Since $b$ is smooth, $F$ is a smooth vector field, and $\psi_Rb$ is smooth with support in $\overline{B_{2R}}$, so every integral below is taken over a compact set and is finite. Thus, the left-hand side of the lemma reads  $\int_{\R^3}\curl F\cdot\psi_Rb\,dx$. 

Recall that, for a smooth vector field $F$ and a smooth compactly supported vector field $G$, we have $\operatorname{div}(F\times G)=G\cdot\curl F-F\cdot\curl G$.
Since $F\times G$ is smooth with compact support, the divergence theorem on a ball $B_\rho\supset\operatorname{supp}G$ shows that the integral of its divergence vanishes, and then, we can write 
\[
\int_{\R^3}\curl F\cdot G\,dx=\int_{\R^3}F\cdot\curl G\,dx .
\]
We apply this with $G=\psi_Rb$. By the product rule for the curl, $\curl(\psi_Rb)=\psi_R\curl b+\nabla\psi_R\times b$, and we obtain
\begin{equation}\label{eq:forhr}
\int_{\R^3}\curl F\cdot\psi_Rb\,dx=\int_{\R^3}\psi_R\,\bigl((\curl b)\times b\bigr)\cdot\curl b\,dx+\int_{\R^3}\bigl((\curl b)\times b\bigr)\cdot(\nabla\psi_R\times b)\,dx .
\end{equation}
The first integral vanishes, because the cross product $(\curl b)\times b$ is orthogonal to each of its factors, and in particular to $\curl b$. 

To compute the second integral we use the identity
\[
(A\times B)\cdot(C\times D)=(A\cdot C)(B\cdot D)-(A\cdot D)(B\cdot C),\qquad A,B,C,D\in\R^3 
.
\]
Choosing $A=\curl b$, $B=D=b$ and $C=\nabla\psi_R$, so that $B\cdot D=|b|^2$ and $B\cdot C=b\cdot\nabla\psi_R$, we get
\begin{multline*}
\int_{\R^3}
\bigl((\curl b)\times b\bigr)\cdot(\nabla\psi_R\times b) \ dx 
\\ =
\int_{\R^3} (\curl b\cdot\nabla\psi_R)\,|b|^2-(\curl b\cdot b)\,(b\cdot\nabla\psi_R)
\ dx 
=H_R
.
\end{multline*}
Thus, gathering together this last expression with \eqref{eq:forhr}, we obtain the desired identity.
\end{proof}
 
\subsection{Proof of lemma \ref{lemma.Hall.decay}}

Before proceeding  with the  proof of the decay of the Hall term,   we establish a preliminary lemma concerning the boundary behavior of the Caffarelli–Silvestre  extension. 
Indeed, since the Hall term involves the gradient of $b$ on the boundary $\{y=0\}$, it is essential to specify in what sense $b^* $ and  its tangential gradient $\nabla_x b^*$  attain their boundary values. 
 The following lemma provides the required convergence properties.

\begin{lemma}\label{lemma.ext}
Let $\sigma\in(0,2)$ and let $v\in C^\infty(\mathbb R^3) \cap  L^\infty(\mathbb R^3)$. Set $v^*(x,y):=(P_y^{(\sigma)}*v)(x)$ for $y>0$.
 Then,
\begin{enumerate}
\item[(i)] $v^*\in C^\infty(\mathbb R^3\times(0,\infty))$ and $\|v^*(\cdot,y)\|_{L^\infty}\le\|v\|_{L^\infty}$ for every $y>0$. Moreover, if $v\in L^s(\mathbb R^3)$ for some $s\in[1,\infty]$, then $\|v^*(\cdot,y)\|_{L^s}\le\|v\|_{L^s}$ for every $y>0$.
\item[(ii)] For every compact set $K\subset\mathbb R^3$,
\[
\sup_{x\in K}\bigl(|v^*(x,y)-v(x)|+|\nabla_xv^*(x,y)-\nabla v(x)|\bigr)\xrightarrow[y\to0]{}0 .
\]
\end{enumerate}
\end{lemma}

\begin{proof}
We first observe that the Poisson kernel satisfies the scaling
identity
\[
P_y^{(\sigma)}(x)=y^{-3}P_1^{(\sigma)}(x/y).
\]
In particular, $\{P_y^{(\sigma)}\}_{y>0}$ is an approximate
identity: each kernel is nonnegative, has integral one, and
\[
\lim_{y\downarrow0}
\int_{|x|>\delta}P_y^{(\sigma)}(x)\,dx=0
\qquad\text{for every }\delta>0.
\]

To prove (i), we note first that, for every compact interval
$J\subset(0,\infty)$ and every multi-index $\gamma$ in the
variables $(x,y)$, the explicit formula  for the kernel  gives
\[
\sup_{y\in J}
\bigl\|D^\gamma P_y^{(\sigma)}\bigr\|_{L^1(\mathbb R^3)}
<\infty.
\]
In particular, for every compact interval $J\subset(0,\infty)$ and every multi-index $ \gamma$ 
in the variables $(x,y)$ there is $C>0$ such that  
$ |D^\gamma P_y^{(\sigma)}(x)|\le  C(1+|x|)^{-3-\sigma }$ 
for all $x\in\mathbb R^3$ and $y\in J$. Since  $(1+|x-z|)^{-3-\sigma}
\le
 2^{3+\sigma}(1+|x_0-z|)^{-3-\sigma} $ whenever $|x-x_0|\le1$, and the right-hand side is integrable in $z$, the integrand  $D^\gamma[P_y^{(\sigma)}(x-z)] \ v(z) $ is dominated, for $ ( x , y) $ near any given point, by an integrable function independent of $ (x, y)$. 
 Thus, since $v $ is bounded,   differentiation  under  the integral
sign is  therefore justified by dominated convergence.
Then, iterating this argument yields
\[
D^\gamma v^*( x, y)
=
\int_{\mathbb R^3}
 D^\gamma\! 
 \left[P_y^{(\sigma)}(x-z)\right]v(z)\,
dz
\]
for every such multi-index $ \gamma$. 
Consequently, we can write
\[
v^*\in C^\infty(\mathbb R^3\times(0,\infty)).
\]
Furthermore, note that the positivity and normalization of the kernel imply
\[
|v^*(x,y)|
\leq
\int_{\mathbb R^3}P_y^{(\sigma)}(x-z)|v(z)|\,dz
\leq \|v\|_{L^\infty(\mathbb R^3)}.
\]
Thus, for $ 1 \leq s <\infty $, considering  Young's convolution inequality we get 
\[
\|v^*(\cdot,y)\|_{L^s}
\leq \|P_y^{(\sigma)}\|_{L^1}\|v\|_{L^s}
=\|v\|_{L^s}.
\]
Then, gathering together with the preceding $L^\infty$ estimate, we prove 
(i).

We now prove (ii). To this end, we fiix a compact set $K \subset \mathbb R^3 $, 
we choose a bounded open neighborhood $U$ of $K$ and a function
$\eta\in C_c^\infty(\mathbb R^3)$ such that $\eta=1$ on $U$.
With this, we write
\[
v=f+g,\qquad f:=\eta v,\qquad g:=(1-\eta)v.
\]
Then, we have $f\in C_c^\infty(\mathbb R^3)$, while $g$ is bounded
and vanishes on $U$. In particular, there exists $d>0$ such
that
\[
\operatorname{dist}(K,\operatorname{supp}g)\geq d.
\]
Now,  since $ \{P_y^{(\sigma)}\}_{ y >0} $ is an approximate identity
and $f \in C_c^\infty (\mathbb R^3)$, we have
\[
P_y^{(\sigma)}*f\longrightarrow f
\quad\text{uniformly on }\mathbb R^3.
\]
On the other hand, considering that  differentiation  commutes with convolution for
$ f \in C_c^\infty(\mathbb R^3)$, the same argument applied
to each first derivative of $f$ yields
\[
\nabla_x(P_y^{(\sigma)}*f)
=P_y^{(\sigma)}*\nabla f
\longrightarrow \nabla f
\quad\text{uniformly on }\mathbb R^3.
\]
It remains to estimate the contribution of $g$.
Note that, for $ x \in K$, the support condition gives
$|x-z|\geq d$ whenever $g(z)\neq0$. 
Hence, we can write 
\[
\begin{aligned}
|(P_y^{(\sigma)}*g)(x)|
&\leq \|g\|_{L^\infty}
\int_{|w|\geq d}P_y^{(\sigma)}(w)\,dw\\
&=\|g\|_{L^\infty}
\int_{|w|\geq d/y}P_1^{(\sigma)}(w)\,dw
\longrightarrow0
\end{aligned}
\]
as $y\downarrow0$, uniformly for $x\in K$.
Now, note that, by part (i) applied to $g=(1-\eta)v\in C^\infty(\mathbb R^3)\cap L^\infty(\mathbb R^3)$, we have
\[
\nabla_x(P_y^{(\sigma)}*g)(x) 
=
\int_{\mathbb R^3}\nabla P_y^{(\sigma)}(x-z)\,g(z)\,dz 
,
\]
and, for the tangential gradient, direct differentiation of the
kernel gives
\[
|\nabla_xP_y^{(\sigma)}(w)|
\leq C_\sigma y^\sigma |w|^{-4-\sigma},
\qquad w\neq0.
\]
Thus, since $|x-z|\geq d$ on the support of the integrand, we obtain
\[
\begin{aligned}
|\nabla_x(P_y^{(\sigma)}*g)(x)|
&\leq C_\sigma\|g\|_{L^\infty}y^\sigma
\int_{|w|\geq d}|w|^{-4-\sigma}\,dw\\
&\leq C_{\sigma,d}\|g\|_{L^\infty}y^\sigma
\longrightarrow0
\end{aligned}
\]
as $y\downarrow0 $, again uniformly  for $ x \in K$.
Finally, note that  $f=v$ and $\nabla f=\nabla v$ on $K$.
Thus,  combining
the preceding estimates for $f$ and $g$ proves
\[
\sup_{x\in K}
\left(
|v^*(x,y)-v(x)|
+
|\nabla_xv^*(x,y)-\nabla v(x)|
\right)\longrightarrow0
\]
as $y\downarrow0$. This establishes (ii) and completes
the proof.
\end{proof}

\begin{proof}[Proof of Lemma \ref{lemma.Hall.decay}]
To begin, note that, since $b\in\dot H^{\beta/2}(\R^3)$, Theorem \ref{thm:CS} (with $\sigma=\beta$) gives 
\[
E=C_\beta^{-1}\|(-\Delta)^{\beta/4}b\|_{L^2}^2<\infty,
\]
and the Sobolev embedding gives $b\in L^q(\R^3)$. We recall that $\nu=1-\beta\in(-1,0]$. We fix $R>1$ and we write $\psi=\psi_R$ and $\chi=\chi_R$. 
By Lemma \ref{lemma.ext}(i), we have $|b^*|\le  M$ in $\R^3\times(0,\infty)$, and $\|b^*(\cdot,y)\|_{L^q}\le\|b\|_{L^q}$. 
To continue we define
\[
W:=|b|^2\nabla\psi-(b\cdot\nabla\psi)\,b,\qquad 
 W^*:=|b^*|^2 \nabla \psi-  (b^*\cdot\nabla\psi)\,b^*,
\]
so that $H_R=\int_{\R^3}\curl b\cdot W\,dx$ by Lemma \ref{lemma.Hall.id}.

\smallskip
For $(x,y)\in\R^3\times(0,\infty)$ we also set
\[
g(x,y):=\chi_R(y)\,\curl_xb^*(x,y)\cdot W^*(x,y).
\]
Then, by Lemma \ref{lemma.ext}(i), $g$ is smooth, it vanishes for $|x|\ge2R$ (because $\nabla\psi$ is supported in $\mathcal C_R$), and for $y\ge2R$.
Now,  we fix $\varepsilon\in(0,1)$. Thus, by the fundamental theorem of calculus in the variable $y$ and Fubini's theorem on the compact set $\overline{B_{2R}}\times[\varepsilon,2R]$, we have
\begin{equation}\label{eq:Hall.FTC}
\int_{\R^3}g(x,\varepsilon)\,dx=-\int_\varepsilon^{2R}\int_{\R^3}\p_yg(x,y)\,dx\,dy .
\end{equation}
Thus,  by Lemma \ref{lemma.ext}(ii)  applied on the compact set $\overline{\mathcal C_R} $, 
we have $b^*(\cdot,\varepsilon) \to b$ and $\nabla_xb^*(\cdot,\varepsilon)\to\nabla b$ uniformly on 
$\overline{\mathcal C_R}$. 
Moreover, $\chi_R(\varepsilon)=1$ because $\varepsilon<R $. 
Hence, we have that  $g(\cdot,\varepsilon)\to\curl b\cdot W$ uniformly on $\overline{\mathcal C_R}$, and since all functions vanish outside $\mathcal C_R$, we can write 
\[
H_R=\lim_{\varepsilon\to0}\int_{\R^3}g(x,\varepsilon)\,dx .
\]
Thus, it is enough to bound the right-hand side of \eqref{eq:Hall.FTC} uniformly in $\varepsilon$.

\smallskip
We have
\[
\p_yg=\chi_R'\,\curl_xb^*\cdot W^*+\chi_R\,\curl_x\p_yb^*\cdot W^*+\chi_R\,\curl_xb^*\cdot\p_yW^* .
\]
For fixed $y>0$, the  vector field $ W^*(\cdot, y)$ is
 smooth with compact support,  so we can 
  integrate by parts in the second term, to write 
\[
\int_{\R^3}\chi_R\,\curl_x\p_yb^*\cdot W^*\,dx=\int_{\R^3}\chi_R\,\p_yb^*\cdot\curl_xW^*\,dx .
\]
Next, using $\curl(fG)=\nabla f\times G+f\curl G$ and $\curl\nabla\psi=0$, we have
\[
\curl_xW^*=\nabla_x|b^*|^2\times\nabla\psi
-
\nabla_x(b^*\cdot\nabla\psi)\times 
b^*-(b^*\cdot\nabla\psi)\curl_xb^*
,
\]
with
 $\nabla_x(b^*\cdot\nabla\psi)=(\nabla_xb^*)^T
 \nabla\psi + (\nabla^2\psi)b^*$. 
 Hence, considering $\p_yW^*=2(b^*\cdot\p_yb^*)\nabla\psi-(\p_yb^*\cdot\nabla\psi)b^*- ( b^* \cdot\nabla\psi )\p_y 
 b^*$,  we obtain the pointwise bounds
\begin{align}
|W^*|&\le2|b^*|^2|\nabla\psi|,\label{eq:W1}\\
|\curl_xW^*|&\le  C\bigl(|\nabla_xb^*|\,|b^*|\,|\nabla\psi|+|b^*|^2|\nabla^2\psi|\bigr),\label{eq:W2}\\
|\p_yW^*|&\le  C\,|b^*|\,|\p_yb^*|\,|\nabla\psi|,\label{eq:W3}
\end{align}
Now, by using $|\curl_xb^*|\le2|\nabla_xb^*|$, $|\nabla_xb^*|+|\p_yb^*|\le2|\bar\nabla b^*|$, and $|\chi_R'|\le  C/R$ with $\chi_R'$ supported in $[R,2R]$, we deduce from \eqref{eq:Hall.FTC} that
\begin{equation}\label{eq:Hall.J}
\Bigl|\int_{\R^3}g(x,\varepsilon)\,dx\Bigr|\le  C\,(J_1+J_2+J_3),
\end{equation}
where
\[
J_1:=\frac1{R^2}\int_R^{2R}\!\!\int_{\mathcal C_R}|\nabla_xb^*|\,|b^*|^2dx\,dy,\qquad
J_2:=\frac1R\int_0^{2R}\!\!\int_{\mathcal C_R}|\bar\nabla b^*|^2\,|b^*|\,dx\,dy,
\]
\[
J_3:=\frac1{R^2}\int_0^{2R}\!\!\int_{\mathcal C_R}|\p_yb^*|\,|b^*|^2dx\,dy .
\]
Note
Note that the bound \eqref{eq:Hall.J} does not depend on $\varepsilon$, since the integrals over $(\varepsilon,2R)$ have been enlarged to integrals over $(0,2R)$.

\smallskip
By H\"older's inequality in $x$ with exponents $\frac q2$ and $\frac q{q-2}$, by Lemma \ref{lemma.ext}(i), and since $1-\frac2q=\frac\beta3$, we have for every $y>0$
\begin{equation}\label{eq:aux}
\int_{\mathcal C_R}|b^*(x,y)|^2dx\le\|b^*(\cdot,y)\|_{L^q}^2\,|\mathcal C_R|^{1-\frac2q}\le  C\|b\|_{L^q}^2R^{\beta}.
\end{equation}

\smallskip
Here we use $\beta\ge1$. Indeed, since $\nu=1-\beta\le0$, the function $y\mapsto y^{-\nu}=y^{\beta-1}$ is non-decreasing, so that $y^{-\nu}\le(2R)^{\beta-1}$ for $0<y\le2R$. Therefore
\[
|\bar\nabla b^*|^2=y^{-\nu}\,y^{\nu}|\bar\nabla b^*|^2\le(2R)^{\beta-1}\,y^\nu|\bar\nabla b^*|^2\qquad\text{for }0<y\le2R,
\]
and, using $|b^*|\le  M$,
\[
J_2\le\frac{M}{R}(2R)^{\beta-1}\int_{\mathbb R^4_+}y^\nu|\bar\nabla b^*|^2dx\,dy\le  C\,M\,E\,R^{\beta-2}.
\]

\smallskip
By the Cauchy--Schwarz inequality with the splitting $|\p_yb^*|=y^{\nu/2}|\p_yb^*|\cdot y^{-\nu/2}$,
\[
J_3\le\frac1{R^2}\Bigl(\int_{\mathbb R^4_+}y^\nu|\p_yb^*|^2\Bigr)^{1/2}\Bigl(\int_0^{2R}\!\!\int_{\mathcal C_R}y^{-\nu}|b^*|^4dx\,dy\Bigr)^{1/2}.
\]
Now, note that 
\[
|\partial_y b^*|^2
\le  |\nabla_x b^*|^2+|\partial_y b^*|^2
=|\bar\nabla b^*|^2.
\]
Thus, multiplying it by  $ y^{1-\beta}>0 $  and integrating over $
\mathbb R^4_+$, we get
\[
\int_{\mathbb R^4_+}
y^{1-\beta}|\partial_y b^*|^2\,dx\,dy 
\le
\int_{\mathbb R^4_+}
y^{1-\beta}|\bar\nabla b^*|^2\,dx\,dy 
=E
.
\]
On the other hand, using $ | b^* |^4
\le
 M^2|b^*|^2 $, \eqref{eq:aux} and $-\nu=\beta-1>-1$, we have
\begin{multline*}
\int_0^{2R}\!\!\int_{\mathcal C_R}y^{-\nu}|b^*|^4dx\,dy\le  M^2\int_0^{2R}y^{\beta-1}\int_{\mathcal C_R}|b^*(x,y)|^2dx\,dy
\\ \le
 CM^2\|b\|_{L^q}^2R^\beta\int_0^{2R}y^{\beta-1}dy=CM^2\|b\|_{L^q}^2R^{2\beta}
 ,
\end{multline*}
where we used $\beta>0$ to integrate $y^{\beta-1}$ near $y=0$. Therefore, we can write 
\[
J_3\le  C\,M\,\|b\|_{L^q}\,E^{1/2}R^{\beta-2}.
\]
To estimate of $J_1$ we proceed in the same way. Indeed, by applying Cauchy--Schwarz with the same weight, we get 
\[
\begin{aligned}
J_1 \le
 \frac1{R^2}
& \left(\int_{\mathbb R^4_+}
y^{1-\beta}|\nabla_xb^*|^2\,dx\,dy\right )^{1/2}\\
&\times
\left(\int_R^{2R} \int_{ \mathcal C_R} 
y^{\beta-1}|b^*|^4\,dx\,dy\right)^{1/2}.
\end{aligned}
\]
Then, we stress the fact that the first integral in the right-hand side is bounded by 
$E$. The second one can be  estimated by extending 
 $(R,2R) $ to  $ (0, 2R)$ arguing as before. 
  Thus,  we obtain 
\[
J_1\le
 CM\|b\|_{L^q}E^{1/2}R^{\beta-2}.
\]

\smallskip
Gathering the preceding estimates and letting $\varepsilon\to0$ in \eqref{eq:Hall.FTC}--\eqref{eq:Hall.J}, we get
\[
|H_R|\le   C\,M\,\bigl(E+\|b\|_{L^q}E^{1/2}\bigr)R^{\beta-2},
\]
and $H_R\to0$ as $R\to\infty$, since $\beta<2$.
\end{proof}

 \medskip

\paragraph{\bf Datasets} Data sharing does not apply to this article as no datasets were generated or analyzed during the
current study.

\paragraph{\bf  Conflict of interest} In addition, the authors declare no conflict of interest and confirm that they contributed equally to this paper.

%------------------
%------ end of document
%------------------

 \end{document}